\documentclass[12pt,reqno]{amsart}

\usepackage{amssymb}
\usepackage{amscd}
\usepackage{amsfonts}
\usepackage{setspace}
\usepackage{version}
\usepackage{setspace}
\usepackage{version}

\numberwithin{equation}{section}
    \theoremstyle{plain}
\newtheorem{theorem}[subsection]{Theorem}
\newtheorem{prototheorem}[subsection]{Prototype Theorem}
\newtheorem{proposition}[subsection]{Proposition}

\newtheorem{conjecture}[subsection]{Conjecture}

\theoremstyle{definition}
\newtheorem{definition}[subsection]{Definition}
\newtheorem{problem}[subsection]{Problem}

\theoremstyle{remark}

\newtheorem{example}{Example}
\renewcommand{\leq}{\leqslant}
\renewcommand{\geq}{\geqslant}
\newsavebox{\proofbox}
\savebox{\proofbox}{\begin{picture}(7,7)  \put(0,0){\framebox(7,7){}}\end{picture}}

\newcommand\Z{\mathbb{Z}}
\newcommand\R{\mathbb{R}}

\newcommand\C{\mathbb{C}}

\newcommand\SL{\operatorname{SL}}

\newcommand\F{\mathbb{F}}

\newcommand\eps{\varepsilon}

\onehalfspace

\begin{document}

\title{Small doubling in groups}

\author{Emmanuel Breuillard}
\address{Laboratoire de Math\'ematiques\\
B\^atiment 425, Universit\'e Paris Sud 11\\
91405 Orsay\\
FRANCE}
\email{emmanuel.breuillard@math.u-psud.fr}

\author{Ben Green}
\address{Centre for Mathematical Sciences\\
Wilberforce Road\\
Cambridge CB3 0WA\\
England }
\email{b.j.green@dpmms.cam.ac.uk}

\author{Terence Tao}
\address{Department of Mathematics, UCLA\\
405 Hilgard Ave\\
Los Angeles CA 90095\\
USA}
\email{tao@math.ucla.edu}


\begin{abstract}
Let $A$ be a subset of a group $G = (G,\cdot)$. We will survey the theory of sets $A$ with the property that $|A\cdot A| \leq K|A|$, where $A \cdot A = \{a_1 a_2 : a_1, a_2 \in A\}$. The case $G = (\Z,+)$ is the famous Freiman--Ruzsa theorem.
\end{abstract}

\maketitle

\onehalfspace

\setcounter{tocdepth}{1}	

\tableofcontents

\section{Small doubling in abelian groups}\label{sec1}  Let  $G =(G,+)$ be an abelian group, the group operation being written with the $+$ symbol. If $A \subseteq G$ is a finite set, we may consider the sumset $A+A := \{a_1+a_2: a_1, a_2 \in A \}$.
We have the trivial bounds
\begin{equation}\label{triva}
 |A| \leq |A+A| \leq \min( \textstyle\frac{1}{2}|A|(|A| + 1), |G| )
 \end{equation}
on the cardinality $|A+A|$ of this sumset.  One expects the trivial upper bound to be attained with equality (or near-equality) unless $A$ has some special additive structure. For example, it is certainly attained when $A = \{1,2, 2^2,\dots, 2^{n-1}\}$ consists of powers of two.

Clarifying what exactly is meant by \emph{special additive structure} turns out to be very interesting, and is the main topic of this survey. Specifically, we will be interested in describing as carefully as we can the structure of non-empty finite sets $A$ for which $\sigma[A] := |A + A|/|A|$ is at most $K$, where $K \in \R^+$ is some constant. We say that such a set $A$ has \emph{doubling} at most $K$. If $\sigma[A]$ is ``small'', we informally say that $A$ has small doubling.

Let us begin with some examples of sets with small doubling. The simplest example is that of a finite subgroup, or a subset of one.

\begin{example}\label{ex1}
Suppose that $A$ is a finite subgroup of $G$. Then $|A + A| = |A|$, and so $\sigma[A]= 1$. Similarly if $A$ is a coset of a subgroup of $G$. If $A$ is not a whole subgroup but occupies a non-zero proportion $\delta$ of some finite subgroup $H \leq G$ then $A + A \subseteq H$, and so $\sigma[A] \leq 1/\delta$. \end{example}

It is a nice exercise to prove that the \emph{only} finite non-empty sets with doubling $1$ are cosets of subgroups.  On the other hand it is very easy to come up with an example of a set $A$ with small doubling which is not related to a subgroup.

\begin{example}\label{ex2}
Suppose that $G = \Z$, and let $h_0, u_1$ and $N_1 >0$ be integers.  We define the arithmetic progression\footnote{The notation here may seem odd. The point is that an arithmetic progression is a very special case of a much more general object called a \emph{coset nilprogression}, which we will discuss in detail in what follows using an elaboration of the same notation.}
$$P(u_1; N_1) := \{h_0 + n_1 u_1 : 0 \leq n_1 < N_1\}$$
and take $A$ to be the arithmetic progression
\[ A = h_0 + P(u_1; N_1) = \{h_0 + n_1 u_1 : 0 \leq n_1 < N_1 \}.\]
Then $A + A$ is the arithmetic progression
$$ A+A = 2h_0 + P(u_1; 2N_1-1) = \{2h_0 + n_1u_1 : 0 \leq n_1 < 2N_1 - 1\},$$
and hence $\sigma[A] \leq 2$. If $A$ occupies a proportion $\delta$ of some arithmetic progression then $\sigma[A] \leq 2/\delta$.
\end{example}

There are multidimensional constructions of a similar nature.

\begin{example}\label{ex3}
Suppose that $G = \Z$, let $h_0,u_1,\ldots,u_d$ and $N_1,\ldots,N_d >0$ be integers.  We introduce the $d$-dimensional progression
\begin{align*} P(u_1,\dots, u_d; N_1, \dots N_d) &:=  \{ n_1 u_1 + \dots + n_d u_d : 0 \leq n_i < N_i\}\end{align*}
and then define
\begin{align*} A & :=h_0 + P(u_1,\dots, u_d; N_1, \dots N_d) \\
& =  \{h_0 + n_1 u_1 + \dots + n_d u_d : 0 \leq n_i < N_i \}.\end{align*}
The sumset $A + A$ is then the progression
\begin{align*}
A+A &= 2h_0 + P(u_1,\dots,u_d,2N_1-1,\dots,2N_d-1)\\
&= \{ 2h_0 + n'_1 u_1 + \dots + n'_d u_d : 0 \leq n'_i < 2N_i - 1\}.
\end{align*}
Thus $A+A$ can be covered by $2^d$ translates of $A$, so that $\sigma[A] \leq 2^d$. If $A$ occupies a proportion $\delta$ of some  $d$-dimensional progression then $\sigma[A] \leq 2^d/\delta$.   % this is true even without properness
\end{example}

Finally, one can combine any of these examples using a direct product construction.

\begin{example}\label{ex4}
Suppose that $A_1 \subseteq G_1$, $A_2 \subseteq G_2$ and that $\sigma[A_i] \leq K_i$ for $i = 1, 2$.  Consider $A_1 \times A_2$ as a subset of $G_1 \times G_2$. Then $\sigma[A_1 \times A_2] \leq K_1 K_2$.
\end{example}

It turns out that, qualitatively at least, the above four examples provide a complete description of sets with small doubling in abelian groups. In the case $G = \Z$ this was established by Freiman \cite{freiman-book} and Ruzsa \cite{ruzsa-sumset}.

\begin{theorem}[Freiman's theorem]\label{thm1.1}  Let $A$ be a finite non-empty set of integers $\Z$ such that $\sigma[A] \leq K$.  Then $A$ is contained within a \emph{generalised arithmetic progression}
\begin{align*} h_0 + P &(u_1,\ldots,u_r;  N_1,\ldots,N_r)\\ &  := \{ h_0 + n_1 u_1 + \ldots + n_r u_r: n_1,\ldots,n_r \in \Z, 0 \leq n_i < N_i \}.\end{align*}
Here, $h_0$ and $u_1,\ldots,u_r \in \Z$ are integers, the \emph{rank} $r$ is\footnote{this notation means a quantity bounded by $C_K$ for some $C_K$ depending only on $K$.} $O_K(1)$ and the \emph{volume} $V := N_1 \dots N_r$ satisfies\footnote{To write $X \ll_K Y$ means that $X \leq C_K Y$ for some $C_K$ depending only on $K$. We might equivalently write $X = O_K(Y)$.} $V \ll_K |A|$.
\end{theorem}

Note that $\Z$ does not have any interesting subgroups, so only Examples \ref{ex2} and \ref{ex3} are relevant here. At the other (high-torsion) extreme, Ruzsa \cite{ruzsa-finite-field} gave a very short and elegant proof of the following statement. Here, $\F_2^\omega$ is the direct product of countably many copies of the finite field $\F_2$.

\begin{theorem}[Ruzsa]\label{thm1.2}
Let $A$ be a finite non-empty subset of $\F_2^{\omega}$, and suppose that $\sigma[A] \leq K$. Then there exists a subgroup $H$ containing $A$ such that $|H| \ll_K |A|$.
\end{theorem}

Ruzsa's theorem works in $\F_p^{\omega}$ for an arbitrary prime $p$, although the dependence of the $\ll_K$ constant is not uniform\footnote{The optimal value of this constant was worked out recently in \cite{lovett}.} in $p$. Ruzsa and the second author \cite{green-ruzsa} combined these two results to get a result valid for all abelian groups.

\begin{theorem}[Green--Ruzsa]\label{ftag}  Let $A$ be a finite non-empty subset of an additive group $G$ such that $\sigma[A] \leq K$.  Then there exists a \emph{coset progression} $H+P$, where $H$ is a finite subgroup of $G$ and $P = P(u_1,\ldots,u_r;$ $N_1,\ldots,N_r)$ is a generalised arithmetic progression of rank $O_K(1)$, such that $A \subseteq H+P$ and $|H| N_1 \ldots N_r \ll_K |A|$.
\end{theorem}

These theorems completely resolve the qualitative question of describing the structure of sets $A$ whose doubling constant $\sigma[A]$ is at most $K$. There are many very interesting quantitative issues in connection with this question, and we will address these in \S \ref{sec5}.

Let us give a brief selection of other results connected with small doubling in abelian groups.

The first does not concern finite sets (although there are variants of it that do, such as Propositions \ref{prop5.3} and \ref{prop5.4}). If $A \subseteq \R^d$ is open then we define its doubling constant to be $\sigma[A] := \mu(A + A)/\mu(A)$, where $\mu$ is Lebesgue measure.

\begin{proposition}\label{bm}
Suppose that $A$ is an open subset of $\R^d$, and that $\sigma[A] \leq K$.  Then $d \leq \log_2 K$.
\end{proposition}

\begin{proof}
This is essentially trivial: $A + A$ contains the dilate $2 \cdot A$, which has $2^d$ times the volume of $A$.  The claim also follows from
the more general \emph{Brunn-Minkowski inequality} $\mu(A+B)^{1/d} \geq \mu(A)^{1/d} + \mu(B)^{1/d}$ (see e.g. \cite{gardner-brunn-minkowski-survey}).
\end{proof}

There are still further results connected with \emph{very} small doubling, when $K < 2$, or with moderately small doubling (when $2 \leq K \leq 3$ say).  Let us finish this section by giving a very small and incomplete selection of them. Perhaps the most famous is the Cauchy-Davenport-Chowla theorem \cite{cauchy,davenport}.

\begin{theorem}
Suppose that $p$ is a prime and that $A \subseteq \Z/p\Z$. Suppose that $\sigma[A] < 2$. Then $A + A$ is the whole of $\Z/p\Z$ \emph{(}and in particular $|A| > p/2$\emph{)}.
\end{theorem}

See e.g. \cite[Theorem 5.4]{tv-book} for a proof.  Kneser's theorem \cite{kneser} generalises the above inequality to arbitrary abelian groups. A consequence of it is the following.

\begin{theorem}
Suppose that $G$ is an arbitrary abelian group and that $A \subseteq G$. Suppose that $\sigma[A] \leq K$ where $K < 2$. Then $A + A$ is a union of cosets of a subgroup $H \leq G$ of size at least $(2 - K)|A|$.
\end{theorem}

Finally let us mention a result \cite{freiman-book} known as \emph{Freiman's $3k - 3$ theorem}, concerning sets of integers with doubling at most (roughly) $3$. See \cite{lev-smeliansky} for a simpler proof and a generalisation to pairs of sets. It gives a very precise version of Theorem \ref{thm1.1} in this regime.

\begin{theorem}
Suppose that $A \subseteq \Z$ is a finite set with $|A| \geq 3$ and with doubling constant $K = \sigma[A]$. Suppose that $K \leq 3 - \frac{3}{|A|}$. Then $A$ is contained in an arithmetic progression $P$ of length at most $(K-1)|A| + 1$.
\end{theorem}

Results by Stanchescu \cite{stan1,stan2} make various assertions, more precise than Theorem \ref{thm1.1}, for values of $K$ in the range $3 \leq K < 4$.

Finally we remark that in many of the above theorems the hypothesis $\sigma[A] \leq K$ may be varied to other, related, conditions such as $|A - A| \leq K|A|$ or $|A + B| \leq K|A|^{1/2}|B|^{1/2}$ using standard additive combinatorial lemmas; see \cite[Chapter 2]{tv-book}.  There are also variants when one replaces the full sumset $A+A$ by a partial sumset $A \stackrel{G}{+} A := \{ a+b: (a,b) \in G \}$ for some (dense) subset $G$ of $A \times A$, using what is now known as the \emph{Balog-Szemer\'edi-Gowers lemma}.  Again, see \cite[Chapter 2]{tv-book} for details and further references, and \S \ref{sec4} for further comments.

\section{Small doubling in arbitrary groups -- examples}

We now turn to the main focus of this survey, which is to study inverse sumset theorems in the \emph{noncommutative} setting, in which one works with finite nonempty subsets $A$ of an arbitrary group $G$. To emphasise the fact that $G$ is not necessarily abelian, we write the group operation multiplicatively.

We are now interested in the structure of finite sets $A \subseteq G$ with the property that $\sigma[A] := |A \cdot A|/|A|$ is at most $K$, where $A \cdot A := \{a_1a_2 : a_1, a_2 \in A\}$.
The trivial bounds are now $|A| \leq |A \cdot A| \leq \min(|A|^2,|G|)$. Equality can occur in the upper bound, for example if $A = \{ x y^i : i = 0,1\dots, n-1\}$ where $x,y$ are generators of a non-abelian free group.  As in the first section, we begin with some examples. The first few of these are parallel to the abelian examples we discussed before.

\begin{example}
Suppose that $A$ is a subgroup of $G$. Then $|A \cdot A| = |A|$, and $\sigma[A] = 1$. Similarly if $A = Hx$ is a coset of some subgroup $H \leq G$ where $x$ lies in the normaliser of $H$ (that is to say $xH = Hx$). If $A$ is not a whole subgroup but occupies a proportion $\delta$ of some finite subgroup $H \leq G$ then $A \cdot A \subseteq H$, so $\sigma[A] \leq 1/\delta$.
\end{example}

It is a nice exercise to prove the converse to this, namely that the only sets with doubling 1 are cosets $Hx$, where $H$ is a subgroup and $x$ normalises $H$.

\begin{example}
The nonabelian analogue of an arithmetic progression is a geometric progression $P(u_1; N_1) := \{u_1^{n_1} : 0 \leq n_1 < N_1\}$. Assuming all $N_1$ elements are distinct, we have $A \cdot A = \{u_1^{n_1'} : 0 \leq n'_1 < 2N_1\}$, and so $\sigma[A] \leq 2$. If $A$ occupies a proportion $\delta$ of some geometric progression then $\sigma[A] \leq 2/\delta$.
\end{example}

As in the abelian case, there are multidimensional constructions of a similar nature, but one must be a little careful in the absence of commutativity.

\begin{example}\label{ex7}
Let $A$ be a set of the form
$$P(u_1,\dots, u_d; N_1, \dots, N_d) := \{u_1^{n_1} u_2^{n_2} \dots u_d^{n_d} : 0 \leq n_i < N_i\},$$
 where the $u_i$ commute and $N_1,\ldots,N_d>0$ are integers. We call this a $d$-dimensional progression. Then $A \cdot A$ is equal to
$$P(u_1,\dots,u_d;2N_1-1,\dots,2N_d-1) = \{ u_1^{n'_1} u_2^{n'_2} \dots u_d^{n'_d} : 0 \leq n'_i < 2N_i - 1\}.$$
Thus, $A \cdot A$ can be covered by $2^d$ dilates of $A$, so that $\sigma[A] \leq 2^d$. If $A$ occupies a proportion $\delta$ of some (proper) $d$-dimensional progression then $\sigma[A] \leq 2^d/\delta$.
\end{example}

Just as in the abelian case, we may consider direct products of sets with small doubling and thereby obtain new examples. In the non-abelian case, however, there are two genuinely new examples of sets with small doubling such as the following.

\begin{example}\label{ex8} Let $N_1,N_2,N_{1,2}$ be positive integers, and let $A$ be the set of $3 \times 3$ matrices defined as follows. Let
\[ u_1 :=  \left( \begin{matrix} 1 & 1 & 0 \\ 0 & 1 & 0 \\ 0 & 0 & 1 \end{matrix} \right) , \quad u_2 := \left( \begin{matrix} 1 & 0 & 0 \\ 0 & 1 & 1 \\ 0 & 0 & 1 \end{matrix} \right) ,\] and set
\begin{align*}
A &= P(u_1, u_2, [u_1,u_2]; N_1, N_2, N_{1,2}) \\
&:= \{ u_1^{n_1} u_2^{n_2} [u_1, u_2]^{n_{1,2}} : 0 \leq n_1 < N_1, 0 \leq n_2 < N_2, 0 \leq n_{1,2} < N_{1,2}\}.
\end{align*}
Here,
\[ [u_1, u_2] := u_1 u_2 u_1^{-1} u_2^{-1} = \left( \begin{matrix} 1 & 0 & 1 \\ 0 & 1 & 0 \\ 0 & 0 & 1 \end{matrix} \right)\] is the commutator of $u_1$ and $u_2$. Noting that
\[ u_1^{n_1} u_2^{n_2} [u_1, u_2]^{n_{1,2}} = \left( \begin{matrix} 1 & n_1 & n_1n_2 + n_{1,2} \\ 0 & 1 & n_2 \\ 0 & 0 & 1 \end{matrix} \right),\] it follows that  $|A| = N_1N_2 N_{1,2}$. Furthermore one may easily check that
\[ A \cdot A \subseteq  \left\{ \left( \begin{matrix} 1 & n'_1 & n'_{1,2} \\ 0 & 1 & n'_2 \\ 0 & 0 & 1 \end{matrix} \right) : \begin{array}{l} 0 \leq n'_1 < 2N_1\\ 0 \leq n'_2 < 2N_2 \\ 0 \leq n'_{1,2} < 3N_1N_2 + 2N_{1,2}\end{array}\right\}.\] Thus if $N_{1,2} \geq N_1N_2$ then $\sigma[A] \leq 20$.
\end{example}

We call the preceding example a \emph{nilprogression}. The name comes from the fact that the group of $3 \times 3$ upper-triangular matrices (the Heisenberg group) is nilpotent of class $2$, which means that higher-order commutators such as $[u_1, [u_2, u_3]]$ are all equal to the identity. We will define nilprogressions in general later on. The second new type of example combines subgroups with progressions in a manner which is not a direct product. The next example was shown to us by Helfgott.

\begin{example}\label{ex9}
Let $p$ be a large prime, let $r, s, t \in \F_p$ be fixed generators of $\F_p^*$, let $N_1,N_2,N_3$ be positive integers, and define $A$ to be a set of $3 \times 3$ matrices over $\F_p$ as follows. Set
\[ A = H \cdot P(u_1, u_2; u_3; N_1, N_2, N_3)\] where
\[ H := \left\{ \left( \begin{matrix} 1 & x& z \\ 0 & 1 & y \\ 0 & 0 & 1 \end{matrix} \right) : x, y, z \in \F_p\right\},\]
\[ u_1 := \left( \begin{matrix} r & 0 & 0 \\ 0 & 1 & 0 \\ 0 & 0 & 1 \end{matrix} \right), u_2 := \left( \begin{matrix} 1 & 0 & 0 \\ 0 & s & 0 \\ 0 & 0 & 1 \end{matrix} \right),  u_3 := \left( \begin{matrix} 1 & 0 & 0 \\ 0 & 1 & 0 \\ 0 & 0 & t \end{matrix} \right)\] for some $r, s, t \in \F_p^*$, and
\[ P(u_1, u_2, u_3; N_1, N_2, N_3) :=\{ u_1^{n_1} u_2^{n_2} u_3^{n_3} : 0 \leq n_i < N_i\},\] as in Example \ref{ex7} above. Thus
\[ A = \left\{ \left( \begin{matrix} r^{n_1} & x& z \\ 0 & s^{n_2} & y \\ 0 & 0 & t^{n_3} \end{matrix} \right) : x,y,z \in \F_p, 0 \leq n_i < N_i\right\}.\]
One may easily check that
\[ A \cdot A \subseteq \left\{ \left( \begin{matrix} r^{n'_1} & x& z \\ 0 & s^{n'_2} & y \\ 0 & 0 & t^{n'_3} \end{matrix} \right) : x,y,z \in \F_p, 0 \leq n'_i < 2N_i\right\},\] and so $\sigma[A] \leq 8$.\end{example}

Examples \ref{ex8} and \ref{ex9} (and in fact all of the other examples we have mentioned) are \emph{coset nilprogressions}, which turn out to be the appropriate generalisation of a coset progression (cf. Theorem \ref{ftag}) to the nonabelian setting.

To conclude this section we discuss the general notion of a coset nilprogression, that is to say the natural generalisation of all the preceding examples. There are several roughly equivalent definitions, which turn out to be ``essentially the same'', meaning that a coset nilprogression of one type is efficiently covered by coset nilprogressions of another type. We begin by giving one definition of a \emph{nilprogression}, following \cite[Definition 2.5]{bgt-big}.

\begin{definition}[Nilprogression]
Let $u_1,\dots, u_r$ be elements in a nilpotent group of step $s$, that is to say a group in which commutators of order $s+1$ or greater are all trivial. Let $N_1,\ldots,N_r$ be positive integers. Define the \emph{nilprogression} $P^*(u_1, \dots, u_r; N_1, \dots, N_r)$ to consist of all products of the $u_i$ and their inverses $u^{-1}_i$ for which the letter $u_i$ and its inverse $u^{-1}_i$ appear at most $N_i$ times, and the terms in the product may be arranged in an arbitrary order (e.g. $P^*(u_1,u_2;1,1)$ contains $u_1 u_2$, $u_2 u_1$, $u_1^{-1} u_2$, etc.).
\end{definition}

We have written $P^*$ instead of $P$ to emphasise the fact that this is not \emph{quite} the same as the objects $P(u_1,\dots, u_r, N_1, \dots, N_r)$ we considered earlier. It can be shown that if $A$ is a nilprogression then $\sigma[A] \ll_{r,s} 1$. With this in hand it is quite straightforward to define a coset nilprogression.

\begin{definition}[Coset nilprogression]
Let $G$ be a group and suppose that $u_1,\dots u_r \in G$. Suppose that $H$ is a finite subgroup of $G$ which is normal in $G_0 := \langle u_1,\dots, u_r\rangle$. Suppose that $G_0/H$ is nilpotent of step $s$. Then the set $H \cdot P^*(u_1,\dots, u_r; N_1,\dots, N_r)$ is called a \emph{coset nilprogression} of rank $r$ and step $s$.
\end{definition}

Once again one may show that if $A$ is a coset nilprogression then $\sigma[A] \ll_{r,s} 1$, that is to say coset nilprogressions are examples of sets with small doubling constant.
An alternative way to define nilprogressions is as the image of a ``ball'' in a free nilpotent group. This gives objects which generalise our examples more directly, but requires quite a lot of nomenclature concerning basic commutators. For more details see \cite[Definition 1.4]{breuillard-green-nil} and also Tointon's paper \cite{tointon}, which clarifies aspects of the relation between the two types of nilprogression\footnote{Note that Tointon calls the objects of \cite[Definition 1.4]{breuillard-green-nil} \emph{nilpotent progressions} but otherwise his nomenclature is essentially the same as ours.}.

\section{Small doubling in arbitrary groups -- theorems}\label{sec3}

A basic aim in the subject, first suggested by Helfgott and Lindenstrauss, is to show that all sets of small doubling are related to one of the examples discussed in the preceding section, namely coset nilprogressions. Theorems \ref{thm1.1}, \ref{thm1.2} and \ref{ftag} in \S \ref{sec1} were results of this type in the abelian case. While such results have now been established by the authors in full generality, for applications such as the ones in \S \ref{sec6} one often needs a result with good bounds, and in the general case none are known. Much work, then, has been done on specific groups (for example matrix groups) where additional structure is extremely helpful and quite precise results have been obtained. Furthermore one does not always need (in fact one essentially never needs) to see the full structure of a coset nilprogression to draw useful applications.

With the aim of clarity of exposition, we will in this section only examine results of the following type, which we call ``The structure theorem''.

\begin{prototheorem}\label{thm3}
Suppose that $A$ is a set in some group $G$, belonging to a specific class \textup{(}matrix group, nilpotent group, solvable group, free group \dots\textup{)} and that $\sigma[A] \leq K$. Then there is a set $A' \subseteq A$, $|A'| \geq |A|/K'$, with $A'$ contained in some set $P$ lying in a ``structured'' class of sets $\mathcal{C}$.
\end{prototheorem}

These results are, therefore, a little weaker than the theorems of \S \ref{sec1}, which cover the whole of $A$ by a structured object. However theorems of that type can be obtained from results having the form of Prototype Theorem \ref{thm3}, and to an extent this amounts only to additive-combinatorial ``book-keeping'', although some more precise variants of this type are both deep and interesting. \vspace{11pt}

\emph{Doubling less than 2.} The case of very small doubling, in which $\sigma[A]$ is close to $1$, received attention at the hands of Freiman \cite{frei-small} almost 50 years ago (see also \cite{freiman-non-abelian}). He showed (among other things) that if $\sigma[A] < 3/2$, then $H := A \cdot A^{-1}$ is a finite group of order $|A \cdot A| = \sigma[A] |A|$, and that $A \subseteq xH = Hx$ for some $x$. In a similar vein, an argument of Hamidoune \cite{hamidoune, tao-hamidoune} shows that if $\sigma[A] < 2-\eps$ for some $\eps>0$, then there exists a finite group $H$ of order $|H| \leq \frac{2}{\eps} |A|$, such that $A$ can be covered by at most $\frac{2}{\eps}-1$ right-cosets $Hx$ of $H$. See also \cite{sanders-hami} for a different proof of a related result.  Very recently, a more complete classification of the sets $A$ with $\sigma[A] < 2$ was achieved in \cite{devos}.
\vspace{11pt}

\emph{Small doubling in nilpotent and solvable groups.} Results along the lines of Theorem \ref{thm3} with $G$ nilpotent or solvable of fixed step have been developed in various papers \cite{breuillard-green-nil, bg2, fisher-katz-peng,  gill-helfgott,  sanders-monomial, tao-noncommutative, tao-solvable,  tointon}, there being a tradeoff in each case between generality and the quality of the bounds obtained. A quite satisfactory recent result of Tointon \cite{tointon} is the following, which applies to arbitrary nilpotent groups of fixed step.

\begin{theorem}[Tointon] Suppose that $G$ is nilpotent of step $s$. Then the structure theorem holds with $K' \sim \exp(K^{O_s(1)})$ and with $\mathcal{C}$ consisting of coset nilprogressions of rank $K^{O_s(1)}$ and size no more than $\exp(K^{O_s(1)})|A|$.
\end{theorem}

Let us also mention a short note of the authors \cite{bgt-nilpotent-dimension-lemma}, which adapts an argument of Gleason from the theory of locally compact groups to show that the $s$-dependence is unnecessary when $G$ is torsion-free. \vspace{11pt}

\emph{Small doubling in matrix groups.} When $G$ is a group of matrices over a field, one has available a wide variety of machinery. One may attempt to exploit the fact that matrix multiplication involves both addition and multiplication of the entries, and hence bring into play results from \emph{sum-product theory} such as \cite{bkt}. One may also try to involve algebraic geometry and particularly the theory of algebraic groups. All of these techniques have enjoyed success.

The first result about small doubling in matrix groups was due to Elekes and Kir\'aly \cite{elekes-kiraly}.

\begin{theorem}[Elekes-Kir\'aly]\label{elekes-kiraly}
Suppose that $G = \SL_2(\R)$. Then the structure theorem holds for some $K' = O_K(1)$ and with $\mathcal{C}$ consisting of abelian subgroups of $\SL_2(\R)$.
\end{theorem}

One interesting consequence of this is that it implies the same result with $G$ equal to the free group.  This is because $\SL_2(\R)$ contains two elements (and hence $k$ elements, for any $k$) generating a free group. Further important work on additive combinatorics in the free group was done by Razborov \cite{razborov} and Safin \cite{safin}.

Elekes and Kir\'aly did not obtain useful bounds for $K'$.  Subsequent advances have addressed this issue, and have also led to the replacement of $\SL_2(\R)$ by an arbitrary matrix group. Significant progress in this regard was made by Helfgott \cite{helfgott-sl2}, who applied the sum-product theorem of Bourgain, Katz and Tao \cite{bkt} to show\footnote{In fact Helfgott does not quite state this result or the one for $\SL_3(\C)$, his concern having been with $\SL_2(\F_p)$ and $\SL_3(\F_p)$, but it follows very easily from his methods.} that Theorem \ref{elekes-kiraly} holds with $G$ replaced by $\SL_2(\C)$ and with $K'$ having polynomial dependence on $K$.  He subsequently generalised this to $\SL_3(\C)$, with $\mathcal{C}$ consisting of nilpotent subgroups of $\SL_3$. Chang \cite{chang-sl3} also proved various results in this direction, for example obtaining a structure theorem in the case $G = \SL_3(\Z)$ prior to the work of Helfgott.

A breakthrough came in 2009 with a paper of Hrushovski \cite{hrush}, who applied model-theoretic arguments to generalise Elekes--Kir\'aly's result to $\SL_n$.

\begin{theorem}[Hrushovski]\label{hrush}
Suppose that $G = \SL_n(\C)$. Then the structure theorem holds for some $K' = O_K(1)$ and with $\mathcal{C}$ consisting of solvable subgroups of $\SL_n(\C)$.
\end{theorem}

Combining this with the main result of \cite{bg2}, ``solvable'' can be replaced by ``nilpotent''.  More down-to-earth proofs are now known due to work of Pyber-Szab\'o \cite{pyber-szabo} and the authors \cite{bgt,bgt-glasgow}, and furthermore these arguments show that $K'$ can be taken to be $K^{O_n(1)}$. These arguments use some of Helfgott's ideas as well as some more purely algebraic group theoretic facts. The argument of \cite{bgt} was, in addition, heavily influenced by groundbreaking work of Larsen and Pink \cite{larsen-pink} in their work on finite subgroups (as opposed to finite subsets of small doubling) of linear groups.

We have discussed the case $G = \SL_n(\C)$. However the analogous question over finite fields is more interesting, enjoys wider application, and was the historial motivation for much of the work we have just mentioned. In this setting Helfgott's result states the following.

\begin{theorem}[Helfgott]
Suppose that $G = \SL_2(\F_p)$. Then either $|A| \geq K^{-C} |G|$, or else the structure theorem holds with $K'$ polynomial in $K$ and with $\mathcal{C}$ consisting of solvable \textup{(}or upper-triangular\textup{)} subgroups of $\SL_2(\F_p)$.
\end{theorem}

This was generalised to $\F_q$, $q$ arbitrary, by Dinai \cite{dinai}. Subsequently, Helfgott \cite{helfgott-sl3} generalised his previous argument to the setting of $\SL_3(\F_p)$ and finally Pyber-Szab\'o \cite{pyber-szabo} obtained the same result in $\SL_n(\F_q)$, as well as in more general finite simple groups of Lie type and bounded rank, while the authors \cite{bgt} obtained simultaneously closely related results. See the related surveys \cite{bmsri,pyber-szabo-msri}.\vspace{11pt}

\emph{General groups.}  A qualitatively complete classification of sets of small doubling in arbitrary nonabelian groups was obtained by the authors \cite{bgt-big}.

\begin{theorem}\label{bgt-theorem}
Suppose that $G$ is an arbitrary group. Then the structure theorem holds for some $K' = O_K(1)$ and with $\mathcal{C}$ consisting of coset nilprogressions $P$ with rank and step $O_K(1)$, and with $|P| \leq |A|$.
\end{theorem}

In fact, the theorem applies to sets with small doubling in \emph{local} groups, which we will not define here (and in fact this generalisation is an important part of the proof). However the dependence of $K'$ on $K$ is not known explicitly at all.   On the other hand, the rank and step of $P$ can be chosen to be at most $6 \log_2 K$; see \cite[Theorem 10.1]{bgt-big}.

We do not have the space to say much about the proof of this theorem here. A crucial theme is a certain ``correspondence principle'', due to Hrushovski \cite{hrush}, linking sets with small doubling (or rather approximate groups, as discussed in the next section) and \emph{locally compact topological groups}. One may then bring into play the extensive theory, largely developed in the 1950s and before,  of locally compact groups in connection with Hilbert's fifth problem. See \cite{bgt-big,tao-book} for considerably more on this. 

\section{Approximate groups}\label{sec4}

This survey is about sets with small doubling constant. However for technical reasons many papers consider a closely-related object called an \emph{approximate group}, first introduced in \cite{tao-noncommutative}.

\begin{definition}
Let $A$ be a finite set in a group $G$. We say that $A$ is a $K$-approximate group if $A$ is symmetric (that is to say if $a \in A$ then $a^{-1} \in A$), contains the identity and if $A \cdot A$ is contained in $X \cdot A$ for some set $X$ of size at most $K$.
\end{definition}

This notion has some technical advantages over the notion of a set with small doubling; for example, it is immediately clear that the image of a $K$-approximate group under a homomorphism is also a $K$-approximate group, but there are examples (even in the abelian case) which show that the doubling constant $\sigma[\pi(A)]$ of a finite set $A$ under a homomorphism $\pi$ is strictly greater than that for $A$; see \cite[Exercise 2.2.10]{tv-book}. By construction, it is clear that any finite $K$-approximate group $A$ has $\sigma[A] \leq K$.  It is also easily seen that one has control over the higher product sets $A^n := A \cdot A \cdots A$, specifically $|A^n| \leq K^{n-1}|A|$, whereas no such bound is generally available for sets of small doubling\footnote{Although one does have $|A^n| \leq K^n|A|$ in the abelian case, a result of Pl\"unnecke and Ruzsa \cite{plunnecke, ruzsa-plunnecke} for which a very elegant proof was recently provided by Petridis \cite{petridis}.}.

Although sets with small doubling need not be approximate groups, we do have the following converse result, obtained in \cite[Theorem 4.6]{tao-noncommutative}. It asserts in some sense that sets of small doubling are essentially ``controlled'' by approximate groups:

\begin{theorem}  Let $A$ be a finite non-empty subset of a multiplicative group $G = (G,\cdot)$ such that $\sigma[A] \leq K$, where $K \geq 2$.  Then one can find a $K^{O(1)}$-approximate group $H$ in $G$ of cardinality $ \ll K^{O(1)} |A|$ such that $A$ can be covered by $K^{O(1)}$ left- or right-translates of $H$.\end{theorem}

The arguments used to prove this are fairly elementary, and are non-commutative variants of some arguments of Ruzsa \cite{ruzsa-sumset}. By contrast the majority of the structural results discussed in \S \ref{sec3} rely on considerable machinery. Thus, in a certain sense, one should think of the theory of sets with small doubling and the theory of approximate groups as equivalent.

\section{Quantitative aspects}\label{sec5}

Thus far we have talked mainly about qualitative results concerning sets with small doubling, with the notable exception of Helfgott's work and its successors, where one obtains polynomial dependencies. A particularly acute example of this is Theorem \ref{bgt-theorem}, where no explicit bounds are known at all.

Even in (in fact \emph{especially} in) the abelian case, quantitative issues are very interesting. We mention some of these now, deferring to the excellent recent survey \cite{sanders-survey} for considerably more detail.

For the most part we discuss the quantitative issues related to Ruzsa's Theorem \ref{thm1.2}, concerning sets with small doubling in $\F_2^{\omega}$. This most abelian of all settings has acted as a significant test case for ideas. Theorem \ref{thm1.2} stated that if $A \subset \F_2^{\omega}$ is a finite set with $\sigma[A] \leq K$ then there exists a subgroup $H \leq \F_2^{\omega}$ containing $A$ with $|H| \leq F(K)|A|$. Ruzsa himself obtained the bound $F(K) \leq K^2 2^{K^4}$. This was subsequently refined by Green-Ruzsa \cite{green-ruzsa} and then by Sanders, and after that by Green-Tao \cite{green-tao-f2}, who showed that one can take $F(K) \leq 2^{2K + o(K)}$, a bound which is sharp up to the $o(K)$ term. This was further refined by Konyagin \cite{konyagin}, and finally Chaim Even-Zohar \cite{evan-zohar} obtained the precise value of $F(K)$. The techniques here are those of extremal combinatorics, specifically the technique of \emph{compressions}.

It was already realised by Ruzsa, however, that trying to cover $A$ by a subgroup is an inefficient endeavour. He attributes to Katalin Marton the following question, which has since become known as the \emph{Polynomial Freiman-Ruzsa Conjecture} (PFR).

\begin{conjecture}
Suppose that $A \subseteq \F_2^{\omega}$ is a set with $\sigma[A] \leq K$. Then $A$ is covered by $K^C$ translates of some subspace $H \leq \F_2^{\omega}$ with $|H| \leq |A|$.
\end{conjecture}

At present this conjecture is unresolved, although there has been spectacular recent progress by Sanders \cite{sanders-freiman}, building on work of Schoen \cite{schoen}. Sanders shows that this conjecture does hold with $K^C$ replaced by $\exp (\log^{4 + o(1)} K)$. % ({\bf was there a subsequent improvement by Konyagin?}).
Ruzsa formulated a number of equivalent statements to this conjecture, which were written up in \cite{green-fin-field}. Perhaps the most attractive is the following statement concerning \emph{almost homomorphisms}.

\begin{conjecture}\label{pfr-2}
Suppose that $f : \F_2^n \rightarrow \F_2^{n'}$ is a map with the property that $f(x+y) - f(x) - f(y) \in S$ for all $x,y$, where $S$ is some set of size $K$. Is it true that $f = \tilde f + g$, where $\tilde f : \F_2^n \rightarrow \F_2^{n'}$ is linear, and $|\mbox{\emph{im}} (g)| \leq K^C$?
\end{conjecture}

This statement is trivial with $K^C$ replaced by $2^K$ (define $\tilde f$ to equal $f$ on a basis of $\F_2^n$). Sanders's result establishes that it holds with $\exp (\log^{4 + o(1)} K)$.

Let us turn now to the structure of sets with small doubling in Euclidean spaces, and in particular\footnote{In fact the theory of small doubling in $\R^d$ and that in $\Z$ are essentially equivalent, since any subset of $\R^d$ is \emph{Freiman isomorphic} to a set of integers.} in $\Z$. In this setting extra tools coming from geometry can be brought into play. In particular we have the following result \cite[Lemma 1.13]{freiman-book}, known as Freiman's lemma.

\begin{proposition}[Freiman's lemma]\label{prop5.3}
Suppose that $A$ is a finite subset of some Euclidean space $\R^d$, and that $\sigma[A] \leq K$. Suppose that $A$ is not contained in any proper affine subspace of $\R^d$. Then
\[ d \leq K - 1 + \frac{d(d+1)}{2|A|}.\]In particular, if $|A| \gg_{\eps, d} 1$ then $d \leq K - 1 + \eps$.
\end{proposition}
The proof of this result is very short, but makes crucial use of convexity. A writeup may be found in, for example, \cite{green-edinburgh}.  A follow-up to this is the next result, also originally due to Freiman \cite{freiman-book}, with a simplified proof given subsequently by Bilu \cite{bilu}.

\begin{proposition}[Freiman-Bilu lemma]\label{prop5.4}
Suppose that $A$ is a finite subset of some Euclidean space $\R^d$, and that $\sigma[A] \leq K$. Then there is a subset $A' \subseteq A$, $|A'| \gg_{\eps,K} |A|$, which is contained in an affine subspace of $\R^d$ of dimension at most $\log_2 K + \eps$.
\end{proposition}

Moderately good dependencies in this theorem are now known \cite{green-tao-freiman-bilu}.
There is a version of the Polynomial Freiman-Ruzsa conjecture for subsets of $\R^d$, a question closely related to that of finding the correct dependencies on $K$ in Proposition \ref{prop5.4}. We are not certain that this has been stated in the literature before.

\begin{conjecture}\label{pfr-z}
Suppose that $A$ is a finite subset of some Euclidean space $\R^d$, and that $\sigma[A] \leq K$. Then $A$ can be covered by $K^C$ translates of some generalised progression $P = P(u_1,\dots, u_r; N_1,\dots, N_r)$ with $r = O(\log K)$.
\end{conjecture}
For a slightly more cautious conjecture one might take, instead of the ``box-like'' generalised progression $P$, a set obtained from the set of lattice points in a convex body in $\R^r$. Whether this is really a more general statement seems slightly unclear and perhaps deserves to be clarified (it is an issue in the geometry of numbers).

Much work has been done on quantitative results in $\Z$, starting with Ruzsa's work and the important paper of Chang \cite{chang-freiman}. A comprehensive history and summary of results may be found in Sanders \cite{sanders-survey}.

Green and Tao \cite{green-tao-equivalence} and independently Lovett \cite{lovett}, building on work of Gowers \cite{gowers-ap4}, demonstrated a fairly tight equivalence between Conjectures \ref{pfr-2} and \ref{pfr-z} and quantitative versions of the inverse conjectures for the Gowers $U^3$-norm. This has yet to reveal itself a viable way to attack these Conjectures \ref{pfr-2} and \ref{pfr-z}, since the only known strategies for proving the inverse conjectures for the $U^3$-norm either \emph{use} results about approximate subgroups of $\Z$, or else are essentially qualitative in nature.

Almost no work has been done on quantitative questions in general groups. It may be, for all we know, that Theorem \ref{bgt-theorem} holds with rather good quantitative dependencies.

\section{Applications and Open Questions}\label{sec6}

We conclude this survey by briefly mentioning some applications of the theory of sets with small doubling (or, usually more accurately, the theory of approximate groups), in various contexts. We encourage the reader to look for more.\vspace{11pt}

\emph{Expanders.} Helfgott's paper \cite{helfgott-sl2} was soon followed by an application due to Bourgain and Gamburd \cite{bourgain-gamburd} concerning the construction of \emph{expanders}. We offer a very brief discussion: for more details, see \cite{lubotzky-survey,tao-expander}. In particular Bourgain and Gamburd's results have now been very substantially generalised, culminating in an almost final result of Varj\'u \cite{varju} and Salehi-Golsefidy and Varj\'u \cite{salehi-varju}.

For the purposes of this brief discussion an expander graph is a $2k$-regular graph $\Gamma$ on $n$ vertices for which there is a constant $c > 0$ such that for any set $X$ of at most $n/2$ vertices of $\Gamma$, the number of vertices outside $X$ which are adjacent to $X$ is at least $c|X|$.  Expander graphs share many of the properties of random regular graphs, and this is an important reason why they are of great interest in theoretical computer science (and would have been of interest to Paul Erd\H{o}s, one imagines). There are many excellent articles on expander graphs ranging from the very concise \cite{sarnak-expander} to the seriously comprehensive \cite{hlw, lubotzky-survey}.

A key issue is that of constructing explicit expander graphs, and in particular that of constructing \emph{families} of expanders in which $k$ and $c$ are fixed but the number $n$ of vertices tends to infinity. Many constructions have been given, and several of them arise from Cayley graphs. Let $G$ be a finite group and suppose that $S = \{g_1^{\pm 1},\dots,g_k^{\pm 1}\}$ is a symmetric set of generators for $G$. The Cayley graph $\mathcal{C}(G,S)$ is the $2k$-regular graph on vertex set $G$ in which vertices $x$ and $y$ are joined if and only if $x y^{-1} \in S$. Such graphs provided some of the earliest examples of expanders \cite{lps,margulis-expander}. A natural way to obtain a family of such graphs is to take some large ``mother'' group $\tilde G$ admitting many homomorphisms $\pi$ from $\tilde G$ to finite groups, a set $\tilde S \subseteq \tilde G$, and then to consider the family of Cayley graphs $\mathcal{C}(\pi(\tilde G),\pi(\tilde S))$ as $\pi$ ranges over a family of homomorphisms. The work under discussion concerns the case $\tilde G = \SL_2(\Z)$, which of course admits homomorphisms $\pi_p : \SL_2(\Z) \rightarrow \SL_2(\F_p)$ for each prime $p$. For certain sets $\tilde S \subseteq \tilde G$, for example
\[ \tilde S = \left\{ \begin{pmatrix} 1 & 1 \\ 0 & 1\end{pmatrix}^{\pm 1}, \begin{pmatrix} 1 & 0 \\ 1 & 1\end{pmatrix}^{\pm 1}\right\}\] or
\[ \tilde S = \left\{ \begin{pmatrix} 1 & 2 \\ 0 & 1\end{pmatrix}^{\pm 1}, \begin{pmatrix} 1 & 0 \\ 2 & 1\end{pmatrix}^{\pm 1}\right\},\]
spectral methods from the theory of automorphic forms may be used to show that $(\mathcal{C}(\pi_p(\tilde G),\pi_p(\tilde S)))_{\mbox{\scriptsize $p$ prime}}$ is a family of expanders. See \cite{lubotzky-survey} and the references therein. These methods depend on the fact that the group $\langle \tilde S \rangle$ has finite index in $\tilde G = \SL_2(\Z)$ and they fail when this is not the case, for example when \begin{equation}\label{123} \tilde S = \left\{ \begin{pmatrix} 1 & 3 \\ 0 & 1\end{pmatrix}^{\pm 1}, \begin{pmatrix} 1 & 0 \\ 3 & 1\end{pmatrix}^{\pm 1}\right\}.\end{equation} In \cite{lubotzky-123question} Lubotzky asked whether the corresponding Cayley graphs in this and other cases might nonetheless form a family of expanders, the particular case of \eqref{123} being known as his ``1-2-3 question''. The paper of Bourgain and Gamburd under discussion answers this quite comprehensively.

\begin{theorem}[Bourgain--Gamburd]\label{bg-theorem}
Let $\tilde G = \SL_2(\Z)$ as above and suppose that $\tilde S$ is a finite symmetric set generating a free subgroup of $\SL_2(\Z)$. Then $(\mathcal{C}(\pi_p(\tilde G),\pi_p(\tilde S)))_{\mbox{\scriptsize $p$ \textup{prime}}}$ is a family of expanders.
\end{theorem}

We shall only say a very few words about the proof. It is well-known (see \cite{hlw}) that proving expansion is equivalent to showing mixing in time $\sim C\log |G|$ of the random walk on $S$, or in other words showing that the convolution power $\mu_S^{n}$ is highly uniform for $n \sim C \log |G|$. Here, $\mu_S := \frac{1}{2k} \sum_{i=1}^k (\delta_{g_i} + \delta_{g_i^{-1}})$. This is analysed in three stages: the \emph{early stage}, where $n \leq c \log |G|$, the \emph{middle stage} where $c \log |G| \leq n \leq \frac{1}{10}C \log |G|$, and the \emph{late stage} where $\frac{1}{10} C \log |G| \leq n \leq C \log |G|$. During the early stage the walk behaves in a very tree-like manner, and in particular at the end of that stage it has already visited a reasonable fraction of $G$. The input from the theory of sets with small doubling/approximate groups comes during the middle stage: here, one must show that the walk does not get ``stuck'', and that by time $\frac{1}{10} C \log |G|$ it has filled out a large portion of $G$. The crucial point is to show, in a certain sense, that  it is impossible to have $\mu^{(n)} \approx \mu^{(2n)}$. If this \emph{did} happen then the support of $\mu^{(n)}$ would behave very much like a set with doubling $\approx 1$, a scenario that can be ruled out using Helfgott's classification of such sets. Finally, a little representation theory is used in the analysis of the \emph{late stage}, specifically the fact that  $\SL_2(\F_p)$ is \emph{quasirandom} in the sense of Gowers \cite{gowers-quasirandom}, that is to say has no nontrivial irreducible representations of small dimension. This observation was first employed in a related context by Sarnak and Xue \cite{sarnak-xue}.

There are many open problems connected with expanders, and we refer the reader to the literature cited above. Let us just mention one (well-known) question which we hope Paul Erd\H{o}s (who wrote several foundational papers in probabilistic group theory) would have liked.

\begin{problem}
Suppose that $k$ elements $g_1, g_2, \dots, g_k$ are selected at random from the alternating group $A_n$, and set $S := \{g_1^{\pm 1}, \dots ,g_k^{\pm 1}\}$. Is it true that, almost surely as $n \rightarrow \infty$, $S$ gives an expander with expansion constant $\eps = \eps(k) > 0$?
\end{problem}

It could be the case that this is so even for $k = 2$. By a result of Dixon \cite{dixon}, $g_1$ and $g_2$ do almost surely generate $A_n$ as $n \rightarrow \infty$, certainly a prerequisite for expansion. By a tour de force result of Kassabov \cite{kassabov}, there does exist $k = O(1)$ and generators $g_1,\dots, g_k$ of $A_n$ for which $S$ gives an expander with $\eps = \eps(k) > 0$, but these are not random generators. Finally, we note that in the case of the alternating group none of the three parts of the Bourgain-Gamburd argument goes through in their current form. In particular, classification of sets with small doubling in $A_n$ (which would be needed with good quantitative bounds) is just as hard as the classification of sets with small doubling in general.\vspace{11pt}

\emph{Gromov's theorem and Varopoulos's result.} There is a close link between sets with small doubling and Gromov's theorem on groups of polynomial growth. Suppose that a group $G$ is generated by a finite symmetric set $S$ (thus $S = S^{-1}$). We say that $G$ has  \emph{polynomial growth} if there are $C$ and $d$ such that $|S^n| \leq Cn^d$ for all $n$. Gromov \cite{gromov} proved that a group has polynomial growth if and only if it is virtually nilpotent, that is to say if and only if some finite index subgroup of it is nilpotent. The link between Gromov's result and sets with small doubling is that infinitely many of the ``balls'' $S^n$ will have $\sigma[S^n] < 2^d + 1$. This is very easy to see: if not, then by induction we have $|S^{2^k}| \geq c(2^d + 1)^{k}$, which is a contradiction for large $k$. By elaborating slightly on this idea, one may fairly easily show that the general theorem for sets with small doubling, Theorem \ref{bgt-theorem}, implies Gromov's theorem. Conversely, large parts of the proof of Theorem \ref{bgt-theorem} are motivated by Gromov's argument, in particular the use of ultrafilters to construct a locally compact group (which closely parallels the Wilkie and van der Dries \cite{wilkie-vandendries} construction of the \emph{asymptotic cone} of a finitely-generated group).

Theorem \ref{bgt-theorem} allows for some strengthenings of Gromov's theorem. For example (\cite[Theorem 1.13]{bgt-big}) one need only assume that $|S^n| \leq Cn^d$ for \emph{one} value of $n > n_0(C,d)$.  Several other such results are given in Section 11 of \cite{bgt-big}, where some applications to differential geometry are also discussed. A possibility, not yet realised, is that a proper understanding of sets with small doubling could be used to study groups of polynomial growth from a quantitative viewpoint. In particular the following conjecture of Grigorchuk \cite{grigorchukgap,grigorchukICM} remains wide open.

\begin{problem}
Is there some constant $c$ (perhaps even $c = \frac{1}{2}$) such that the following is true: if $G$ is generated by a symmetric set $S$, and if $|S^n| \leq e^{n^c}$ for all large $n$, is it true that $G$ has polynomial growth?
\end{problem}

Famous examples of Grigorchuk \cite{grigorchuk} show that this is not true for all $c < 1$. So far, the best result known is due to Shalom and Tao \cite{shalom-tao}, who show that if $|S^n| \leq n^{(\log \log n)^c}$ for large $n$ then $G$ has polynomial growth. (This result does not, however, make use of the connection with approximate groups.) See \cite{grigorchukgap} for the state of the art on the above problem regarding special classes of groups. \vspace{11pt}

One lovely application of Gromov's theorem is the following result of Varopoulos \cite{varopoulos} (see also \cite{woess}).

\begin{theorem}
Let $G$ be a group generated by a finite set $S$. Suppose that the \textup{(}simple\textup{)} random walk with generating set $S$ is recurrent. Then $G$ is finite or has a finite-index subgroup isomorphic to $\Z$ or $\Z^2$.
\end{theorem}

Although Varopoulos's theorem uses Gromov's theorem, it actually only uses that theorem for some value of $d > 2$. It seems, however, that no simpler proof of the theorem is known in that special case. From the point of view of sets with small doubling, one is interested in statements about sets $A$ with $\sigma[A] < 4 + \eps$. However, no analysis of this case is known which is simpler than the general analysis of \cite{bgt-big}.\vspace{11pt}

\emph{An open problem.} We conclude with a very simply-stated open question. We said very little about the proof of the classification of approximate groups in general \cite{bgt-big}. An important ingredient in it (used to establish the correspondence between approximate groups and locally compact groups) was the following result of Croot-Sisask \cite{croot-sisask} and Sanders \cite{sanders}.

\begin{theorem}
Suppose that $A$ is a $K$-approximate group. Then there is a set $S$, $|S| \gg_K |A|$, such that $S^8 \subseteq A^4$.
\end{theorem}

\begin{problem}
In the preceding theorem, can we take $|S| \gg K^{-O(1)} |A|$?
\end{problem}

This is open even in the abelian case.\vspace{11pt}

\noindent\textsc{Acknowledgments.} EB is supported in part by the ERC starting grant 208091-GADA. BG is supported by an ERC starting grant. TT is supported by a grant from the MacArthur Foundation, by NSF grant DMS-0649473, and by the NSF Waterman award.

\bibliographystyle{abbrv}
\bibliography{bibfile}

\end{document}